\documentclass[11pt]{amsart}
\usepackage[utf8]{inputenc}
\usepackage{amssymb,amsmath}
\usepackage{graphicx}
\usepackage{color}
\usepackage{epstopdf}
\usepackage{tikz}
\usepackage{pgfplots}
\usepackage{enumitem}
\usepackage[leftcaption]{sidecap}
\usepackage{float}
\usepackage{algorithm}
\usepackage{setspace}
\usepackage{algpseudocode}
\usepackage{mathtools}
\usepackage[hidelinks]{hyperref}
\usepackage{algorithm}
\usepackage{algpseudocode}
\usepackage{xfrac}

\usepackage{accents}
\usepackage{multicol} 
\usepackage{soul}
\usepackage{subcaption}
\newtheorem{theorem}{Theorem}

\theoremstyle{definition}

\newtheorem{example}[theorem]{Example}

\newtheorem{lemma}[theorem]{Lemma}

\theoremstyle{remark}
\newtheorem{remark}[theorem]{Remark}

\newtheorem{assumption}[theorem]{Assumption}
\numberwithin{theorem}{section}
\numberwithin{equation}{section}
\numberwithin{table}{section}
\numberwithin{figure}{section}
\def\Vo{\V_{\ker}}

\def\Vc{\V_\text{c}}

\def\cHo{\cH_{\ker}}

\def\calAo{\calA_{\ker}}
\def\N{\mathbb{N}}

\def\R{\mathbb{R}}

\def\calA{\mathcal{A}}
 
\def\calB{\mathcal{B}}

\def\cH{\mathcal{H}}

\def\calN{\mathcal{N}}

\def\Q{\mathcal{Q}}

\def\V{\mathcal{V}}
\def\calV{\V}

\def\eps{\varepsilon}

\definecolor{blau}{RGB}{0, 51, 255}
\definecolor{hellblau}{RGB}{153, 204, 255}
\definecolor{hellrot}{RGB}{255, 0, 0}
\definecolor{firebrick}{RGB}{176, 34, 34} 
\definecolor{deep_pink}{RGB}{255, 20, 147} 
\definecolor{sky_blue}{RGB}{74, 112, 139}
\definecolor{slate_blue}{RGB}{71, 60, 139}
\definecolor{chartreuse}{RGB}{118, 238, 0}
\definecolor{chartreuseL}{RGB}{228, 255, 150}
\definecolor{light_blue}{RGB}{178, 223, 238}
\definecolor{dodge_blue}{RGB}{17, 78, 138}
\definecolor{code_backg}{RGB}{238, 216, 174}
\definecolor{myBlue1}{RGB}{101,149,239}  
\definecolor{myBlue2}{RGB}{113,104,238} 
\definecolor{myBlue3}{RGB}{30,144,255} 
\definecolor{myGreen1}{RGB}{154,204,50} 
\definecolor{myGreen2}{RGB}{69,169,0} 
\definecolor{myGreen3}{RGB}{154,205,50} 
\definecolor{myGreen4}{RGB}{105,139,34} 
\definecolor{myRed1}{RGB}{210,105,30} 
\definecolor{myRed2}{RGB}{165,42,42} 
\definecolor{myRed3}{RGB}{139,26,26} 
\definecolor{myLGray}{RGB}{225,225,225} 
\definecolor{mycolor1}{rgb}{0.00000,0.44700,0.74100}%
\definecolor{mycolor2}{rgb}{0.85000,0.32500,0.09800}%
\definecolor{mycolor3}{rgb}{0.92900,0.69400,0.12500}%
\definecolor{mycolor4}{rgb}{0.49400,0.18400,0.55600}%
\definecolor{mycolor5}{rgb}{0.46600,0.67400,0.18800}%
\definecolor{mycolor6}{rgb}{0.30100,0.74500,0.93300}%
\definecolor{mycolor7}{rgb}{0.63500,0.07800,0.18400}%

\DeclareMathOperator{\id}{id}

\DeclareMathOperator{\image}{im}

\newcommand{\dx}{\ensuremath{\, \mathrm{d}x }}
\newcommand{\ds}{\ensuremath{\, \mathrm{d}s }}

\allowdisplaybreaks
\begin{document}
\title{Probabilistic Time Integration for Semi-explicit PDAEs$^*$} 
\author[]{R.~Altmann$^\dagger$, A.~Moradi$^{\dagger}$}
\address{${}^{\dagger}$ Institute of Analysis and Numerics, Otto von Guericke University Magdeburg, Universit\"atsplatz 2, 39106 Magdeburg, Germany}
\email{\{robert.altmann, afsaneh.moradi\}@ovgu.de}
\thanks{${^*}$ This project is funded by the Deutsche Forschungsgemeinschaft (DFG, German Research Foundation) through the project 446856041.\\ 
This article will be published in {\em Statistics and Computing}. }
%
%
\date{\today}
\begin{abstract}
This paper deals with the application of probabilistic time integration methods to semi-explicit partial differential--algebraic equations of parabolic type and its semi-discrete counterparts, namely semi-explicit differential--algebraic equations of index~$2$. The proposed methods iteratively construct a probability distribution over the solution of deterministic problems, enhancing the information obtained from the numerical simulation. Within this paper, we examine the efficacy of the randomized versions of the implicit Euler method, the midpoint scheme, and exponential integrators of first and second order. By demonstrating the consistency and convergence properties of these solvers, we illustrate their utility in capturing the sensitivity of the solution to numerical errors. Our analysis establishes the theoretical validity of randomized time integration for constrained systems and offers insights into the calibration of probabilistic integrators for practical applications.
\end{abstract}
%
%
\maketitle
%
{\tiny {\bf Key words.} PDAEs, semi-explicit systems, probabilistic numerics, randomization} 
%
%
%
\section{Introduction}\label{Sec.intro}
In the last decade, randomized time integration methods for ordinary, partial, and operator differential equations have been studied extensively; see~\cite{chkrebtii2016,ConGSSZ17,lie2019,abd2020,lie22}. These methods aim to statistically quantify the uncertainty introduced by the time discretization. 
For this, classical probabilistic solvers iteratively establish a probability measure over the numerical solution of deterministic initial value problems of the form 
\begin{align}\label{eq:IVP}
    \dot{u}(t) =& f(t,u(t)), \qquad \text{for $0 \leq t \leq T$,}\nonumber\\[-5mm]
    &\\
    u(0) =& u^0 \in\R^d,\nonumber
\end{align}
offering more comprehensive information than a single (deterministic) solve. This means that, instead of yielding just one solution path, the probabilistic approach offers a distribution over possible solutions. Consequently, such methods capture the uncertainty inherent in the numerical solution process, especially in regions where the solution is not directly computed.
Let~$\psi_t\colon [0,T]\times\R^d \rightarrow \R^d$ denote the flow map induced by \eqref{eq:IVP}, meaning that $u(t) = \psi_t(0,u^0)$. On the other hand, we consider a discretization scheme based on a (constant) time step size~$\tau > 0$ with $N \coloneqq T /\tau \in \N$ and corresponding time points $t^n \coloneqq n\tau$ for $n \in [N] \coloneqq \{0, 1,\ldots, N\}$. With this, 
a numerical one-step method can be expressed by a mapping~$\Psi_t\colon[0,T]\times\R^d \rightarrow \R^d$ with  
\[
    u^{ n+1 } = \Psi_{ \tau } (t^n,u^n), \qquad  n \in [N].
\]
Here, $u^n$ denotes the approximation of $u(t^{n}) = \psi_{n\tau}(t^0,u^0)$. Typically, these methods produce a single discrete solution, often accompanied by some form of an error indicator. However, they do not fully quantify the remaining uncertainty in the path statistically. In order to capture the sensitivity of the solution to numerical errors, while maintaining consistent convergence properties from classical deterministic integrators, the probabilistic interpretation of the numerical solution introduced in~\cite{ConGSSZ17} considers the randomized scheme
\[ 
    U^{n+1} 
    = \Psi_{\tau}(t^n,U^n) + \xi^{n}(\tau), \qquad n=0,1,\dots
\]
with $U^0 = u^0$ and appropriately scaled independent and identically distributed (i.i.d.) random variables~$\xi^{n}(\tau)$ with values in $\R^d$. This then results in a sequence of random variables $U^n$ approximating $u(t^n)$. We would like to emphasize that probabilistic integrators may neither inherently provide more accurate solutions than classical deterministic methods, nor are they necessarily computationally cheaper. They prove useful, however, in statistical inference settings, particularly in tasks such as integrating chaotic dynamical systems and solving Bayesian inverse problems~\cite{chkrebtii2016,ConGSSZ17}.

The purpose of this paper is to address the construction and rigorous analysis of probabilistic time integration methods for constrained systems. 
Besides, we consider the operator case leading to so-called partial differential--algebraic equations (PDAEs); see~\cite{LamMT13} for an introduction. To be more precise, we consider a (semi-linear) parabolic system 
\[
    \dot{u}(t) + \calA u(t) = f (t, u)
\]
together with a linear constraint of the form $\calB u(t) = g(t)$ included by the Lagrangian method, cf.~\cite{EmmM13,Alt15}. In the corresponding semi-discretized setting, i.e., after the application of a spatial discretization scheme, this yields a semi-explicit differential--algebraic equation (DAE) of index~$2$; see~\cite[Ch.~VII.1]{HaiW96}. 

Known deterministic time stepping methods for such DAEs and PDAEs include algebraically stable Runge--Kutta methods~\cite{HaiLR89,KunM06,AltZ18}, splitting methods~\cite{AltO17}, (discontinuous) Galerkin methods~\cite{VouR18,AltH21}, and exponential integrators~\cite{AltZ20}. 
%
In this study, we provide a framework how to randomize existing integration schemes, including the specific construction of four randomized methods. We identify assumptions and introduce a local random field, particularly a Gaussian field, to reflect the uncertainty of the solution in between mesh points. The resulting probabilistic solvers allow for repeated sampling in order to explore the solutions uncertainty. In the following example, we illustrate the difficulty arising for constrained systems and the consequences of perturbations affecting the constraint; see also~\cite{AltLM17}.  
\begin{example}\label{ex:FitzHugh}
To illustrate the randomized approach presented in~\cite{ConGSSZ17}, which was originally applied to the unconstrained FitzHugh–Nagumo model~\cite{RamHCC07}, we extend this example by an additional constraint. With the help of the Lagrangian method, we obtain the nonlinear system 
\begin{align*}
	\dot V(t)
	= 3\, V(t) - V^3(t) + 3\, R(t) - \lambda(t), \qquad
	\dot R(t)
	= \frac{ - 5\, V(t) + 1 - R(t) }{15} - \lambda(t)
\end{align*}
with the constraint $V(t) + R(t) = \sin(t)$ and unknowns $V$, $R$, and $\lambda$. As initial conditions, we set $V(0)=-1$ and $R(0)=1$. When a small perturbation is added to the constraint, implemented as a scalar Gaussian random variable $ \eps \sim \calN(0, \sigma ^2) $, this leads to chaotic behavior. As a result, the trajectories strongly deviate from the true solution.  However, if we add perturbations only in the dynamic part, one can see that every probabilistic solution is a good approximation of the exact solution.  
In Figure~\ref{Fig:pert:Fitz:fs}, we present $50$ trajectories for the $V$ species computed by the implicit Euler scheme with~$\sigma = 0.1$, once affecting the constraint and once only acting on the dynamic part of the solution. We refer the readers to Section~\ref{Subsec:PIME} for more details on the computation of the trajectories. This illustrates the sensitivity of the solution to perturbations on the constraint. On the other hand, Figure~\ref{Fig:pert:Fitz:ds} depicts $50$ trajectories for the $V$ species for different noise scales $\sigma$, showing the sensitivity of the solutions with respect to~$\sigma$. Although we employed the same method, namely the probabilistic implicit Euler method, we varied the noise scale to demonstrate that the scale parameter $\sigma$ controls the apparent uncertainty in the solver. This variation affects the magnitude of the error induced by the probabilistic method. At this point, we would like to emphasize that the scale parameter $\sigma$ should, in general, be chosen problem-dependent on basis of specific characteristics of the model. Such a calibration controls the apparent uncertainty in the solver and is discussed later on.
\end{example}
\begin{figure}
 \centering
    \includegraphics[width=0.5\textwidth]{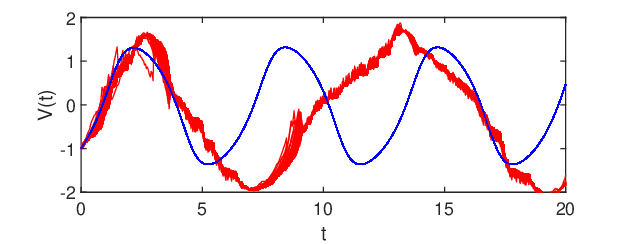}\hfill
        \includegraphics[width=0.5\textwidth]{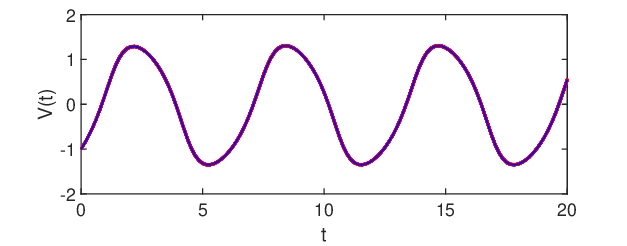}
    \caption{The exact trajectory of $V$ of the constrained FitzHugh–Nagumo model from Example~\ref{ex:FitzHugh} (blue) and $50$ trajectories (red) obtained by the implicit Euler method with Gaussian perturbation on the constraint (left) and on the dynamic part only (right). In both cases the noise scale equals $\sigma = 0.1$. }
    \label{Fig:pert:Fitz:fs} 
\end{figure}
\begin{figure}
\centering    
    \includegraphics[width=0.5\textwidth]{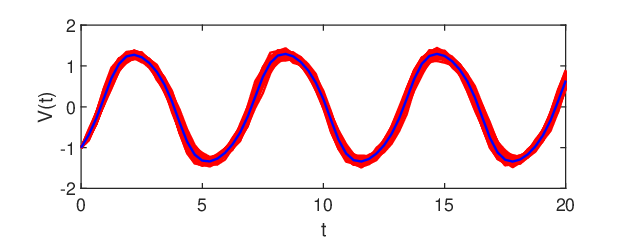}\hfill 
    \includegraphics[width=0.5\textwidth]{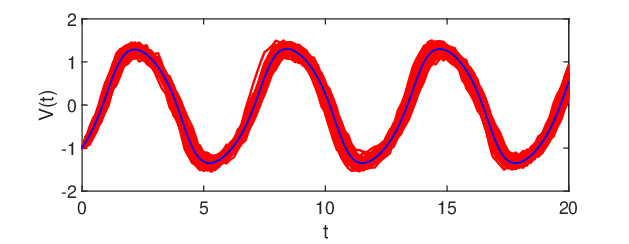}
    \caption{Exact value of~$V$ (blue) and $50$ trajectories (red) in the same setting as in Figure~\ref{Fig:pert:Fitz:fs} with different noise scales, namely $\sigma = 0.5$ (left) and $\sigma=1.5$ (right).}
    \label{Fig:pert:Fitz:ds} 
\end{figure}
In the forthcoming section, we introduce the model of interest and the necessary setup for the construction of probabilistic integrators. The mean square convergence is then analyzed in Section~\ref{sect:convergence}. In Section~\ref{Sec.PNM}, we delve into the construction of four specific integrators, namely the implicit Euler method, the midpoint scheme, and two exponential integrators of first and second order. To confirm the theoretical findings, we provide a numerical example. Finally, Section~\ref{Sec.Calib} is devoted to the calibration of the parameter~$\sigma$ when employing randomized time integrators for constrained parabolic systems. Here, we demonstrate that an appropriate scaling enables the randomized solvers to yield probabilistic solutions without affecting the convergence of the corresponding deterministic scheme.
%
%
\section{Preliminaries}\label{sect:prelim} 
In this section, we present the main assumptions and notations that will be used throughout the remainder of this work. 
%
%
\subsection{Semi-explicit (P)DAEs}\label{sect:prelim:PDAE}
Within this section, we introduce the model of interest, namely (semi-linear) parabolic systems with linear constraints. Note that this is a PDAE, since the corresponding semi-discrete system, which is obtained by the application of a spatial discretization, equals a semi-explicit DAE of index~$2$. More precisely, we consider the following task: seek abstract functions~$u\colon [0,T] \to \V$ and $\lambda\colon [0,T] \to \Q$ such that 
\begin{subequations}
\label{eq:PDAE}
\begin{alignat}{5}
	\dot{u}(t)&\, +\, &\calA u(t)&\, +\, &\calB^\ast \lambda(t)\, 
	&= f(t, u) &&\qquad\text{in } \V^* \label{eq:PDAE:a},\\
	& &\calB u(t)& & 
	&= g(t) &&\qquad\text{in } \Q^* \label{eq:PDAE:b}
\end{alignat}
\end{subequations}
is satisfied for almost every $t\in[0,T]$. Moreover, we assume given initial data~$u(0)=u^0$, which should be {\em consistent} in the sense that~$\calB u^0 = g(0)$. Within this paper, we assume that $\V$ and $\Q$ are given Hilbert spaces. Together with the dual space~$\V^*$, we further assume the existence of a Gelfand triple~$\V, \cH, \V^*$; see~\cite[Ch.~23.4]{Zei90a}. For the involved operators, we mainly follow the assumptions stated in~\cite{AltZ20}. 
\begin{assumption}
\label{assB}
The constraint operator~$\calB\colon \V \to \Q^*$ is linear, continuous, and satisfies an inf--sup condition, i.e., there exists a constant~$\beta>0$ such that 
\[
	\adjustlimits \inf_{q\in\Q\setminus\{ 0 \} }\sup_{v\in\V\setminus\{0\}}
	\frac{\langle \calB v, q\rangle}{\Vert v\Vert_\V \Vert q\Vert_\Q}
	\ge \beta. 
\]
\end{assumption}
\begin{assumption}
\label{assA}
The differential operator~$\calA\colon \V\to\V^*$ is linear, continuous, and elliptic on~$\Vo \coloneqq \ker\calB$, i.e., on the kernel of the constraint operator. Restricted to the kernel, we denote the operator by $\calAo \coloneqq \calA|_{\Vo}\colon \Vo\to\Vo^* \coloneqq (\Vo)^*$. Note that we use $\Vo^*$ as codomain, which means that also the test functions are restricted to $\Vo$. 
\end{assumption}
\begin{remark}
The upcoming results can be extended to the situation where $\calA$ only satisfies a G\aa{}rding inequality on~$\Vo$. In this case, the non-elliptic part can be transferred to the nonlinearity~$f$, cf.~\cite{AltZ20}. 
\end{remark}
\begin{assumption}
\label{assf}
The right-hand sides are assumed to be sufficiently smooth to guarantee the existence of a unique solution. In particular, we assume that $f$ maps into~$\cH$ and that it is Lipschitz continuous in the second component, i.e., there exists a positive constant $L_f$ such that 
\begin{align}
    \label{eq:fLip}
    \big\| f(t,v_1)-f(t,v_2) \big\|_{\cH} 
    \le L_f\, \|v_1-v_2\|_{\cH}.
\end{align}
\end{assumption}
We will also consider corresponding finite-dimensional examples, i.e., semi-explicit DAEs of the form 
\begin{align*}
    \dot u + A u + B^T\lambda 
    &= f(u), \\
    B u \hspace{3.2em}
    &= g.
\end{align*}
Here, the assumptions reduce to $B\in\R^{m,n}$ being of full rank and $A\in\R^{n,n}$ being invertible. 

We would like to mention two examples which fit into the given framework and which will later be considered in the numerical examples.
\begin{example}\label{ex:FitzHugh:revisit}
First, we revisit the constrained FitzHugh--Nagumo model from the introduction. This finite-dimensional example reads 
\begin{align*}
    \begin{bmatrix} \dot V(t) \\ \dot R(t) \end{bmatrix}
    + \begin{bmatrix} -3  & -3 \\ \tfrac13 & \tfrac{1}{15} \end{bmatrix}
    \begin{bmatrix}  V(t) \\ R(t) \end{bmatrix}
    + \begin{bmatrix}  1 \\ 1 \end{bmatrix} \lambda(t) 
    &= \begin{bmatrix} - V(t)^3 \\ \tfrac{1}{15} \end{bmatrix}, \\
    \begin{bmatrix} 1 & 1 \end{bmatrix}
    \begin{bmatrix} V(t) \\ R(t) \end{bmatrix} \hspace{1.88cm}
    &= \sin(t) 
\end{align*}
and is of the considered semi-explicit form. Moreover, the assumed initial conditions~$V(0)=-1$, $R(0)=1$ are consistent with the constraint.  
\end{example}
\begin{example}\label{ex:PDAE}
As a second example, we consider the semi-linear heat equation with a constraint on the integral mean as in~\cite{AltO17}. More precisely, we consider 
\begin{equation}\label{eq:heat:PDAE}
    \dot u - \Delta u + \calB^* \lambda 
    = u^2
    \qquad\text{on }\Omega\coloneqq(0,1)
\end{equation}
with homogeneous Dirichlet boundary conditions for $u$ and the constraint 
\[   
    (\calB u)(t) 
    \coloneqq \int_0^1 u(x,t) \sin(\pi x) \dx = g(t).
\]
Here, the operator~$\calA$ equals the Laplace operator with homogeneous boundary conditions such that Assumption~\ref{assA} is satisfied. Moreover, $\calB$ is inf--sup stable, since it is surjective and maps into a finite-dimensional space.
\end{example}
%
%
\subsection{Decomposition of the solution and flow map}\label{sect:prelim:decomposition}
Within the presented setting, it is possible to decompose the solution $u$ into a dynamic part (which lies in the kernel of~$\calB$) and a complement, which is fully determined by the constraint~\eqref{eq:PDAE:b}. Due to Assumption~\ref{assB}, the constraint operator has a right-inverse~$\calB^-\colon\Q^*\to\V$, leading to  
\[
	\V = \Vo \oplus \Vc \qquad\text{with}\qquad
	\Vo = \ker\calB,\quad 
	\Vc = \image\calB^-.
\]   
Following~\cite{AltZ18}, one may choose $\Vc$, e.g., as the annihilator of $\Vo$. 
%
%
Note that this choice also determines the right-inverse~$\calB^-$. For the resulting decomposition of the solution, we use the notation 
\[
	u = u_{\ker} + u_c, \qquad
	u_{\ker}\colon [0,T] \to \Vo,\ 
	u_c\colon [0,T] \to \Vc.
\]
Inserting this into the constraint~\eqref{eq:PDAE:b}, we get 
\[
	u_c(t) 
	= \calB^- g(t). 
\]
The already mentioned consistency condition of the initial data, $\calB u^0 = g(0)$, determines the value of $u_c(0)$ but not of $u_{\ker}0)$, while $\lambda(0)$ can be computed by~\eqref{eq:PDAE:a}. Hence, we may only prescribe a value for $u_{\ker}(0)$. 
This motivates the definition of a flow map on the kernel, which we denote by~$\psi_t\colon [0,T] \times \Vo \rightarrow \Vo$. For given $t^*\in[0,T]$ with $t^*+t\le T$ and $v_{\ker}\in\Vo$, this map solves the PDAE~\eqref{eq:PDAE} on the time interval $[t^*,t^*+t]$ with initial data $v_{\ker}+\calB^- g(t^*)$. The outcome is then the solution at time $t^*+t$ restricted to the kernel of $\calB$, resulting in $ u_{\ker}(t^{n+1})$, i.e., the solution at time $ t^{n+1} $ within the kernel of $ \calB$,  
\begin{equation}\label{eq:psi}
    u_{\ker}(t^{n+1}) 
    = \psi_\tau(t^n,u_{\ker}(t^n)). 
\end{equation}
%
\subsection{Deterministic and probabilistic setup}\label{sect:prelim:assumptions}
Consider a probability space $(\Omega, \mathcal{F}, \mathbb{P})$ that is rich enough to serve as a shared domain for defining all the random variables and processes being studied and let $\mathbb{E}$ denote the expectation with respect to $\mathbb{P}$. 

By $\Psi_\tau$ we denote a deterministic time integration method which can be used to approximate the solution of the semi-explicit system~\eqref{eq:PDAE} or its semi-discretized counterpart. Based on such a scheme, we define a randomized integrator, resulting in a sequence~$\{U^n\}_{n\in[N]_0}$. Each approximation can again be decomposed into
\[
	U^n = U^n_{\ker} + U^n_c, \qquad
	U^n_{\ker}\in\Vo,\ U^n_{c}\in\Vc,	
\]
with $U^n_c=\calB^-g(t^n)$. As for the continuous solution, we assume that the approximation at time $t^{n+1}$ is fully determined by~$U^n_{\ker}$ and the right-hand sides. Hence, we write 
\begin{equation}\label{eq:Psi}
    U^{n+1}_{\ker} 
    = \Psi_{\tau}(t^n,U^n_{\ker}) + \xi^n(\tau),
\end{equation}
where each $\xi^n$ is assumed to be an i.i.d.~centered Gaussian random variable with values in~$\Vo$. Since the perturbation only acts on the dynamic part of the solution, the constraint $\calB U^{n} = g(t^{n})$ is maintained for all~$n$. Hence, we have $u_c(t^n)=U^n_c$ and the sequence of errors $e^n$ of the random approximation is defined by 
\[
    e^{n}
    \coloneqq u(t^{n}) - U^{n}
    = u_{\ker}(t^{n}) - U_{\ker}^{n},\qquad n\in[N]_0.
\]
Inserting the definition of 
the numerical scheme~\eqref{eq:Psi}, we obtain
\begin{equation}\label{eq:err}
    e^{n+1}
    = u_{\ker}(t^{n+1})-\Psi_{\tau}(t^n,U_{\ker}^n)-\xi^n(\tau),\qquad n\in[N-1]_0.
\end{equation}
Similarly to \cite{ConGSSZ17}, the assumptions on the deterministic scheme $\Psi_\tau$ are summarized in the following. 
\begin{assumption}[Deterministic scheme]\label{ass:Psi}
Assume that the time integration scheme~$\Psi_\tau\colon [0,T]\times \Vo\rightarrow\Vo$ meets the following conditions: 
    \begin{itemize}
        \item[1.] Let $u^{n}_{\ker}$ denote an approximation of $u_{\ker}(t^{n})$ obtained by~$\Psi_\tau$. Then, there exist constants $\tau^*, C, q > 0$ such that for all step sizes~$\tau\in[0,\tau^*]$ and $t^n\in[0,T]$ it holds that
         \begin{equation}\label{eq:errDet}             
             \big\|  u_{\ker}(t^{n})-u^{n}_{\ker}\big\|_\cH 
             \leq C\, \tau^{q}.
         \end{equation}
     \item[2.] There exist constants $\tau^*, L_{\Psi} > 0$ such that for all step sizes~$\tau\in[0,\tau^*]$, $t\in[0,T]$, and $v_1, v_2\in\Vo$ we have the Lipschitz property 
    \begin{equation}\label{eq:lipPsi}
         \big\|\Psi_{\tau}(t,v_1)-\Psi_{\tau}(t,v_2)\big\|_{\cH}
         \leq (1+L_{\Psi}\tau)\, \|v_1-v_2\|_{\cH}.
     \end{equation}
     \end{itemize}
\end{assumption}
Condition~\eqref{eq:errDet} corresponds to the convergence order and is essential for the consideration of any deterministic time integration method $\Psi_\tau$. The second part of Assumption~\ref{ass:Psi} describes a Lipschitz property of the approximate flow map $\Psi_{\tau}$ with respect to its second argument. Next, we state the required assumptions on the random noise, which appears in~\eqref{eq:Psi}. 
\begin{assumption}[Random noise]\label{ass:noise}
Let $ \chi^n(s) $ be a centered Gaussian stochastic process taking values in $ \Vo $, where $ \chi^n(s) \sim \mathcal{N}(0, \sigma^2) $ for all $ s \in [0, t] $ with $\sigma\in\R$.  
The perturbation is given by $ \xi^n(t) = \int_0^t \chi^n(s) \ds$ with $\chi^n$ being i.i.d.~over different time steps and admits parameters $C_{\xi}, p>0$ such that for all $t\in[0,\tau]$ it holds that $\mathbb{E}\, \|\xi^n(t)\|_{\cH}^2 \le C_{\xi}\, t^{2p+1}$.
\end{assumption}
In line with the work~\cite{ConGSSZ17}, within the upcoming convergence analysis, we will observe that choosing $p=q$ preserves the strong order of accuracy of the underlying deterministic integrator. This choice introduces the maximum amount of noise consistent with the accuracy of the original deterministic integrator. 
%
%
\section{Mean Square Convergence Analysis}\label{sect:convergence}
In this section, we analyze the mean square convergence of probabilistic time integration methods for semi-explicit PDAEs of the form~\eqref{eq:PDAE} as introduced in the previous section. We will show that there exists a constant $C > 0$ independent of $\tau$ such that 
\[
    \mathbb{E}\, \|u(t^n)-U^n\|_{\cH}^2 
    \leq C\, \tau^{2r}
\]
for all $n\in[N]$. Therein, $r$ is the {\em mean square order of convergence} of the method. For the unconstrained case, cf.~\cite{ConGSSZ17}, it is known that~$r$ is equal to the minimum of $q$ (convergence order of the deterministic time integrator) and $p$ (order of the random noise). 
Given that our approach is based on a construction in the kernel of the constraint, one can hope to achieve the same order for the given setting. The analysis of this section will make use of the following inequality. 
\begin{lemma}[Discrete Gronwall lemma {\cite[App.~C]{lie22}}]\label{lem:Gronwall}
Suppose that for non-negative sequences $\{x^{{n}}\}_{{n}\in[N]_0}$ and $\{y^{{n}}\}_{{n}\in [N]_0}$ and for some constant $L > 0$, we have $x^{{n}+1} \leq y^{{n}}+(1+\tau L)\, x^{{n}}$ for all $n\in[N-1]_0$. Then, 
\[
    x^{n}
    \leq \bigg(x^0+\sum_{j=0}^{n-1}\, y^j\bigg) \exp\big( L \tau n \big),\qquad {n}\in[N]_0.
\]
\end{lemma}
In the spirit of~\cite[Th.~2.2]{ConGSSZ17}, we can now state the global mean square convergence result for probabilistic time integrators applied to~\eqref{eq:PDAE}. %
\begin{theorem}[Global mean square convergence]\label{thm:meanErr} 
Let Assumptions~\ref{assB}, \ref{assA}, and~\ref{assf} be given.
Consider $\tau\le1$ and let Assumptions~\ref{ass:Psi} and \ref{ass:noise} hold with orders $q, p > 0$, respectively. Then there exists a constant~$C>0$ independent of the step size~$\tau$ such that  
\begin{equation}\label{ineq:global:Err}
     \sup_{n\in[N]} \big(\, \mathbb{E}\,\|u(t^n)-U^{n}\|^2_{\cH}\, \big)^{1/2} 
     \leq C\, \tau^{\min\{p,q\}}.
\end{equation}
\end{theorem}
\begin{proof}
Recall that $U^{n}_{\ker}$ denotes the approximate solution given by the probabilistic solver and let~$u_{\ker}^{n}$ denote the corresponding solution given by the underlying deterministic time stepping scheme. For the latter, we define the error 
\[
    \tilde{e}^{n+1}
    \coloneqq u(t^{n+1})-u^{n+1} 
    = u_{\ker}(t^{n+1}) - u_{\ker}^{n+1}
    = u_{\ker}(t^{n+1})-\Psi_\tau(t^n,u_{\ker}^n). 
\]
Considering the definition of the error $e^{n+1}$ in~\eqref{eq:err} and leveraging the decomposition of the solution $u=u_{\ker}+u_c$, we derive
\begin{align*}
    \|e^{n+1}-\tilde{e}^{n+1}\|^2_\cH 
    &= \big\| \Psi_{\tau}(t^n,u^n_{\ker}) - \Psi_{\tau}(t^n,U^n_{\ker}) - \xi^n(\tau) \big\|^2_\cH \notag \\
    &= \big\|\Psi_{\tau}(t^n,u^n_{\ker}) - \Psi_{\tau}(t^n,U^n_{\ker}) \big\|^2_{\cH} + \|\xi^n(\tau)\|^2_\cH \nonumber\\
    &\qquad\qquad -2\, \big( \Psi_{\tau}(t^n,u^n_{\ker}) - \Psi_{\tau}(t^n,U^n_{\ker}),\xi^n(\tau) \big)_{\cH}\nonumber\\
    &\leq ( 1+L_{\Psi}\tau)^2 \|e^n-\tilde{e}^n\|^2_{\cH}+ \|\xi^n(\tau)\|^2_\cH \nonumber\\
    &\qquad\qquad -2\, \big( \Psi_{\tau}(t^n,u^n_{\ker}) - \Psi_{\tau}(t^n,U^n_{\ker}),\xi^n(\tau) \big)_{\cH},
\end{align*}
where the second part of Assumption~\ref{ass:Psi} was applied in the last step. Now we consider the expectation of this estimate. Since $\xi^n$ equals a zero-mean Gaussian random variable, the expectation of the inner product vanishes and we get together with Assumption~\ref{ass:noise}
\begin{align*}
    \mathbb{E}\, \|e^{n+1}-\tilde{e}^{n+1}\|^2_{\cH}
    \le ( 1+\tilde{L}_{\Psi}\tau )\, \mathbb{E}\, \|e^n-\tilde{e}^n\|^2_{\cH} + C_{\xi}\, \tau^{2p+1},
\end{align*}
where $\tilde{L}_{\Psi}>0$ is chosen such that $(1+L_{\Psi}\tau)^2\leq 1+\tilde{L}_{\Psi}\tau$. Due to $e^0 = \tilde{e}^{0}=0$, an application of the discrete Gronwall lemma  yields 
\begin{align*}
    \mathbb{E}\, \|e^{n}-\tilde{e}^{n}\|^2_{\cH}
    \le TC_{\xi}\,\tau^{2p}\exp\big( \tilde{L}_{\Psi} T \big).
\end{align*}
Finally, an application of the triangle inequality, 
\begin{align*}
    \mathbb{E}\,\| e^{n} \|^2_{\cH}
    &\le 2\, \mathbb{E}\,\|\tilde{e}^{n}\|^2_{\cH} 
    + 2\,\mathbb{E}\,\|e^{n}-\tilde{e}^{n}\|^2_{\cH}, 
\end{align*}
and the assumption on the global error of the deterministic solver given in~\eqref{eq:errDet} yields the desired result. 
\end{proof} 
\begin{remark}
Analogously to \cite{ConGSSZ17}, the convergence result of Theorem~\ref{thm:meanErr} implies that a reasonable option for the noise scale $p$ is to set $p = q$, where $q$ is the order of the underlying deterministic method. This choice of $p$ preserves the convergence of the underlying deterministic method, while providing a probabilistic interpretation of the numerical solution.
\end{remark}
%
%
\section{Examples of Probabilistic Time Integrators}\label{Sec.PNM}
This section is dedicated to exploring randomized time stepping schemes of first and second order for constrained parabolic PDAEs of the form~\eqref{eq:PDAE}. For this, we consider four methods in detail and demonstrate how perturbations can be introduced without affecting the constraints of the underlying model. Furthermore, we show that the randomized numerical solution exhibits convergence properties that are asymptotically no inferior to those of the deterministic numerical solution. This implies that the trajectories obtained from the randomized integrator are all equally valid approximations. 
%
\subsection{Probabilistic implicit Euler scheme}\label{Subsec:PIME}
As a first example, we consider the implicit Euler scheme applied to~\eqref{eq:PDAE}. In the deterministic setting, this reads 
\begin{subequations}
\label{eq:implEuler}	
\begin{align}
	(\id + \tau \calA)\, u^{n+1} - u^n + \tau\calB^*\lambda^{n+1} 
	&= \tau f(t^{n+1},u^{n+1}) \qquad\text{in } \V^*, \label{eq:implEuler:a} \\
	\calB u^{n+1} \hspace{2.83cm}
	&= \phantom{\tau}g(t^{n+1}) \hspace{1.0cm}\qquad\text{in } \Q^*. \label{eq:implEuler:b}
\end{align}
\end{subequations}
Before incorporating perturbations, we discuss Assumption~\ref{ass:Psi}. It is well-known that the implicit Euler scheme converges with order one without constraints~\cite{LubO93}. This then directly translates to the solution part in the kernel of the constraint, i.e., we get~\eqref{eq:errDet} with $q=1$. 
To prove the Lipschitz property~\eqref{eq:lipPsi}, we use once more the decomposition of the iterates~$u^n = u^n_{\ker} + u^n_{c}$ with $u^{n}_{c} = \calB^-g(t^n)$ as described in Section~\ref{sect:prelim:decomposition}. With this, equation~\eqref{eq:implEuler:a} reads 
\begin{align*}
    u^{n+1}_{\ker}+ \calB^- g(t^{n+1}) 
	&= u^n_{\ker} +\calB^- g(t^{n}) - \tau\, \calA\, \big(u^{n+1}_{\ker} +\calB^- g(t^{n+1}) \big) \\
    &\hspace{2.1cm} + \tau\calB^*\lambda^{n+1} + \tau\, f(t^{n+1},u_{\ker}^{n+1}+\calB^-g(t^{n+1})) \quad\text{in } \V^*.
\end{align*}
Restricting this equation to test functions in $\Vo$, the Lagrange multiplier disappears. Moreover, setting $\Vc$ to the annihilator of $\Vo$, also $\calA \calB^- g(t^{n+1})$ vanishes and we get 
\[
    u^{n+1}_{\ker}  
	= u^n_{\ker} + \tau\, \big(- \calA\, u^{n+1}_{\ker} +f(t^{n+1},u_{\ker}^{n+1}+\calB^-g(t^{n+1}))\big) - \calB^-\big( g(t^{n+1}) - g(t^n) \big) 
\]
in $\Vo^*$ and, hence, 
\[
    \Psi_{\tau}(t,v)
    = v + \tau\, \Big[ -\calA\Psi_{\tau}(t,v) + f\big(t+\tau,\Psi_{\tau}(t,v)+\calB^-g(t+\tau)\big) \Big] 
    - \calB^-\big( g(t+\tau) - g(t) \big).
\]
\begin{lemma}
\label{lem:LipschitzEuler}
Let Assumptions~\ref{assB}, \ref{assA}, and~\ref{assf} be given. Then the implicit Euler scheme satisfies the Lipschitz property~\eqref{eq:lipPsi}. 
\end{lemma}
\begin{proof}
Define
\[
    w_1\coloneqq\Psi_{\tau}(t,v_1)-\Psi_{\tau}(t,v_2)\in\Vo,\qquad 
    w_0\coloneqq v_1-v_2\in\Vo.
\]
Rearranging $0\leq \|w_0-w_1\|^2_{\cH}$ leads to 
\begin{align*}
    \tfrac1{2\tau}\, \big( &\|w_1\|^2_{\cH}-\|w_0\|^2_{\cH} \big) \\[2mm]
    &\qquad\leq \tfrac1\tau\, \big( w_1-w_0, w_1\big)_{\cH}\\[2mm]
    &\qquad\leq -\, \big\langle \calA w_1,w_1\big\rangle_{\V^*\times\V} \\[2mm]
    &\qquad\qquad+\left(f(t,\Psi_{\tau}(t,v_1)+\calB^-g(t+\tau))-f(t,\Psi_{\tau}(t,v_2)+\calB^-g(t+\tau)),w_1\right)_{\cH}\\[2mm]
    &\qquad\leq L_f\, \|w_1\|^2_{\cH}.
\end{align*}
The last estimate follows from the ellipticity of $\calA$ on $\calV_{\ker}$ and~\eqref{eq:fLip}. We then obtain
\begin{align*}
    \|\Psi_{\tau}(t,v_1) - \Psi_{\tau}(t,v_2)\|_{\cH}
    \leq (1-2L_f\tau)^{-1/2}\|v_1-v_2\|_{\cH},
\end{align*}
which holds for all $0<\tau\leq\tau^*<\tfrac{1}{2L_f}$. Setting $L_{\Psi}\coloneqq\tfrac{2L_f}{1-2L_f\tau^*}$, we have 
\[
    (1-2L_f\tau)^{-1/2}
    \leq (1-2L_f\tau)^{-1}
    \leq 1+L_{\Psi}\tau
\]
for all $0<\tau\leq \tau^*$, which leads to the desired result. 
\end{proof}
We now turn to the incorporation of perturbations. Here, the precise location of the perturbation is essential in order to ensure that they do not interact with the constraint. Adding the perturbations to the dynamic part, computing one Euler step involves the solution of the saddle point problem 
\begin{equation}\label{eq:ConPDE}
	\begin{bmatrix} \id + \tau \calA & \tau \calB^* \\ \calB & 0\end{bmatrix}
	\begin{bmatrix} U^{n+1} \\ \lambda^{n+1} \end{bmatrix}
	= 
	\begin{bmatrix} U^n + \tau f(t^{n+1},U^{n+1}) + \xi^n(\tau)\\ g(t^{n+1}) \end{bmatrix},    
\end{equation}
where $\xi^n(\tau)\in \Vo$ can be defined as $\sigma\tau^{3/2}\tilde{\xi}^n$ by the natural choice of $p=q=1$. Here, $\sigma$ is a constant noise scale and $\tilde{\xi}^n$ is defined as a random vector where each component is an i.i.d. standard normal random variable, i.e., $\tilde{\xi}_i^n \sim \mathcal{N}(0, 1)$, with $i$ indexing  the components of the vector. The existence of a unique solution to~\eqref{eq:ConPDE} is guaranteed by \cite[Lem.~3.1]{AltZ20}. In Section~\ref{Sec.Calib}, we discuss the case where $\sigma$ is chosen problem-dependent such that the expected order of convergence is maintained. 
\begin{remark}
It is assumed that $\xi^n(\tau) \in \Vo$. Here, however, also~$\xi^n(\tau) \in \V$ is possible, since the difference would only affect the Lagrange multiplier~$\lambda$ and not the constraint.
\end{remark}
\begin{remark}
Note that the perturbation is applied in a way that maintains the structure of the problem, ensuring that the conditions for solvability, such as the ellipticity of $\calA$ on the kernel of $\calB$ and the inf--sup condition for $\calB$, remain valid. This guarantees that adding perturbations does not affect the existence or uniqueness of the solution.
\end{remark}
\begin{remark}
With an implicit time discretization, one needs to solve a possibly nonlinear system in each time step. This is due to the semi-linear term $f(t^{n+1},U^{n+1})$. This computational burden can be addressed by employing a semi-explicit time stepping method, which implements $f(t^n, U^n)$ instead. 
\end{remark}
%
\subsection{Probabilistic midpoint scheme}\label{subs:Midpoint}
The deterministic midpoint scheme applied to the PDAE~\eqref{eq:PDAE} results in 
\begin{subequations}
\label{eq:Midpoint}	
\begin{align}
	\big(\id + \tfrac{\tau}{2} \calA\big)\, u^{n+1} - \big(\id - \tfrac{\tau}{2} \calA\big)\, u^n + \tau\calB^*\lambda^{n+1} 
	&= \tau f^{n+\frac{1}{2}} \qquad\text{in } \V^*, \label{eq:Midpoint:a} \\
	\calB u^{n+1} \hspace{2.83cm}
	&= \phantom{\tau}g(t^{n+1})\hspace{0.87cm}\text{in } \Q^*,
 \label{eq:Midpoint:b}
\end{align}
\end{subequations}
where $f^{n+\frac{1}{2}}=f\left(t^n+\frac{\tau}{2},\frac{1}{2}(u^n+u^{n+1})\right)$. The result is a second-order scheme, i.e., we get~\eqref{eq:errDet} with $q=2$. Using again the fact that $u^{n+1}_c = \calB^-g(t^{n+1})$ and restricting equation~\eqref{eq:Midpoint:a} to test functions in $\Vo$, we obtain 
\[
	\big(\id + \tfrac{\tau}{2} \calA\big)\, u_{\ker}^{n+1}  
	= \big(\id - \tfrac{\tau}{2} \calA\big)\, u^n_{\ker} + \tau f\big(t^n+\tfrac{\tau}{2},\tfrac{1}{2}(u^n+u^{n+1})\big) - \calB^-\big( g(t^{n+1}) - g(t^n)\big)
	\quad\text{in } \Vo^* 
\]
and, therefore,
\[
    \Psi_{\tau}(t,v) 
    = \big(\id-\tfrac{\tau}{2}\calA\big)\, v -\tfrac{\tau}{2} \calA\Psi_{\tau}(t,v) + \tau f\big(t+\tfrac{\tau}{2},\widetilde v(t)\big) - \calB^-\big(g(t+\tau)-g(t)\big)
\]
with $\widetilde v(t)\coloneqq\frac{1}{2}\big(v+\calB^-g(t)+\Psi_{\tau}(t,v)+\calB^-g(t+\tau)\big)$. 
\begin{lemma}
\label{lem:LipschitzMidpoint}
Let Assumptions~\ref{assB}, \ref{assA}, and~\ref{assf} be given. Then the midpoint scheme satisfies the Lipschitz property~\eqref{eq:lipPsi}. 
\end{lemma}
\begin{proof}
Given $v_1,v_2\in\Vo$, we define $w_0$ and $w_1$ as in the proof of Lemma~\ref{lem:LipschitzEuler}. It directly follows that $\widetilde v_1 - \widetilde v_2 = \frac12 (w_0 + w_1)$. This, in turn, leads to  
\begin{align*}
    \tfrac1{\tau}\, \big(\|w_1\|^2_{\cH}-\|w_0\|^2_{\cH}\big)
    &= \tfrac1{\tau}\, \big(w_1-w_0,w_0+w_1\big)_{\cH}\\[2mm]
    &\leq -\tfrac12\, \big\langle\calA (w_0+w_1),w_0+w_1\big\rangle_{\V^*\times\V}\\[2mm]
    &\qquad\qquad+ \Big( f\big(t+\tfrac\tau2,\widetilde v_1\big)-f\big(t+\tfrac\tau2,\widetilde v_2\big), w_0+w_1 \Big)_{\cH}\\[2mm]
    &\leq L_f\, \big( \|w_0\|^2_{\cH}+\|w_1\|^2_{\cH} \big)
\end{align*}
with the Lipschitz constant~$L_f>0$ from~\eqref{eq:fLip}. Therefore, 
\begin{align*}
    \|\Psi_{\tau}(t,v_1) - \Psi_{\tau}(t,v_2)\|_{\cH}
    \leq (1-L_f\tau)^{-1/2}(1+L_f\tau)^{1/2}\, \|v_1-v_2\|_{\cH}
\end{align*}
for all $0<\tau\leq\tau^*<\tfrac{1}{L_f}$. Setting $L_{\Psi}\coloneqq\tfrac{2L_f}{1-L_f\tau^*}$, we obtain
\[ 
    (1-L_f\tau)^{-1/2}(1+L_f\tau)^{1/2}
    \leq 1+L_{\Psi}\tau 
\]
for all $0<\tau\leq \tau^*$, which leads to the desired result. 
\end{proof}
Similarly as in the previous subsection, by adding perturbations to the dynamic part, computing one step of the midpoint scheme involves solving a saddle point problem of the following form
\[
	\begin{bmatrix} \id + \frac{\tau}{2} \calA & \tau \calB^* \\ \calB & \end{bmatrix}
	\begin{bmatrix} U^{n+1} \\ \lambda^{n+1} \end{bmatrix}
	= 
	\begin{bmatrix} (\id - \frac{\tau}{2} \calA)U^n + \tau f^{n+\frac{1}{2}} + \xi^n(\tau)\\ g(t^{n+1}) \end{bmatrix},
\]
where $\xi^n(\tau) \in \Vo$ can be defined as $\sigma\tau^{5/2}\tilde{\xi}^n$ by the natural choice of $p=q=2$, a constant scale~$\sigma$ and i.i.d.~standard normal random variables $\tilde{\xi}^n \sim\calN(0,1)$. 
%
%
\subsection{Probabilistic exponential Euler scheme}\label{subs:exp_Euler}
As a third example, we derive the probabilistic version of the exponential Euler method. The corresponding deterministic scheme for constrained systems has been introduced and examined in~\cite{AltZ20}. Recall that~$\calAo$ denotes the restriction of $\calA$ to the kernel of the constraint operator~$\calB$. Moreover, we introduce $\cHo$ as the closure of $\Vo$ in $\cH$. Due to Assumption~\ref{assA}, we know that $-\calAo$ generates an analytic semigroup on $\cHo$; see \cite[Ch.~7, Th.~2.7]{Paz83}. 

Exponential integrators are based on the variation-of-constants formula, which reads
\begin{align}
    u_{\ker}(t^{n+1})
    &=u(t^{n+1})-\calB^{-}g(t^{n+1}) \nonumber \\
    &= e^{-\tau\calA_{\ker}} \big[u(t^n)-\calB^-g(t^n)\big]
    +\int_{t^n}^{t^{n+1}}e^{-(t^{n+1}-s)\calA_{\ker}} \iota_0\big[f(s,u(s))-\calB^-\dot{g}(s)\big] \ds \label{eq:solformula}
\end{align}
with $\iota_0\colon\cH\equiv\cH^*\rightarrow\cH^*_{\ker}\equiv\cH_{\ker}$ denoting the restriction of test functions (which will be omitted from now on). 
Yet another important tool for the construction is the set of the recursively defined $\varphi$-functions, 
\begin{equation}\label{eq:varphi}
    \varphi_0(z)\coloneqq e^z,\qquad 
    \varphi_{k+1}(z)\coloneqq\dfrac{\varphi_k(z)-\varphi_k(0)}{z}
\end{equation}
with $\varphi_k(0)=1/k!$ for $z=0$. Since $-\calAo$ generates an analytic semigroup, it follows from~\cite[Lem.~2.4]{HocO10} that $\varphi_k(-\tau \calAo)\colon \cHo \to \cHo$, given by 
\[
    \varphi_k(-\tau \calAo)
    = -\tau^{-1} \calAo^{-1}\, \big( \varphi_{k-1}(-\tau \calAo)-\varphi_{k-1}(0) \big),
\]
defines a bounded operator. 

Now, the consideration of suitable quadrature rules in~\eqref{eq:solformula} leads to exponential integrators. The left-hand rule yields the exponential Euler method as introduced in~\cite{AltZ20}. It takes the form 
\begin{align}
	u^{n+1} 
	= \calB^-g(t^{n+1}) + \varphi_0(-\tau\calA_{\ker})\big(u^n - \calB^-g(t^n)\big) + \tau\varphi_1(-\tau\calAo)  \big( f(t^n, u^n) - \calB^- \dot g(t^n) \big). \label{eq:expEuler}
\end{align}
Note that this is an explicit scheme. Considering only the part in the kernel, we get
\begin{equation}\label{eq:PsiExpEul}
       \Psi_{\tau}(t,v) 
    = \varphi_0(-\tau\calA_{\ker})\, v + \tau\varphi_1(-\tau\calAo) \big( f(t, v + \calB^- g(t)) - \calB^- \dot g(t) \big) 
\end{equation}
and it is shown in~\cite{AltZ20} that this method is of order $q=1$. 
\begin{lemma}\label{Lem:LipExpEul}
Suppose that Assumptions~\ref{assB}, \ref{assA}, and~\ref{assf} hold and let $\calA$ be symmetric. Then the exponential Euler scheme satisfies the Lipschitz property~\eqref{eq:lipPsi}. 
\end{lemma}
\begin{proof}
Direct calculations show 
\begin{align*}
    \big\|\Psi_{\tau}(t,v_1)-\Psi_{\tau}(t,v_2)\big\|_{\cH}
    &\le \big\|\varphi_0(-\tau\calA_{\ker})\, (v_1-v_2) \big\|_{\cH} \\
    &\qquad + \tau\, \big\|\varphi_1(-\tau\calA_{\ker})\, \big(f(t, v_1 + \calB^- g(t)) - f(t, v_2 + \calB^- g(t)) \big) \big\|_{\cH} \\
    &\le \big( C_{\varphi_0} + \tau\, C_{\varphi_1} L_f \big)\, \big\| v_1 - v_2 \big\|_{\cH}, 
\end{align*}
where we have used the boundedness of the $\varphi$-functions. Due to the assumed symmetry of $\calA$, \cite[Ex.~2.3]{HocO10} shows that $C_{\varphi_0}\le 1$, which completes the proof.  
\end{proof}
\begin{remark}
Also in the finite-dimensional setting we need that $A\in\R^{n,n}$ is symmetric positive definite in order to guarantee that $C_{\varphi_0}$ is bounded by one. 
\end{remark}
The practical implementation of the exponential Euler method involves the solution of several saddle point problems. This includes the computation of~$x\coloneqq\calB^-g(t^n) \in \Vc \subseteq \V$ (see Section~\ref{sect:prelim:assumptions}), which can be expressed as the solution of the stationary auxiliary problem
\begin{subequations}
	\label{eqn:Binvers}	
	\begin{alignat}{4}
		&\calA x&\ +\ &\calB^* \nu\ &=&\ 0  &&\qquad\text{in } \V^* , \label{eqn:Binvers:a} \\
		&\calB x& & &=&\ g(t^n) &&\qquad\text{in } \Q^*. \label{eqn:Binvers:b}
	\end{alignat}
\end{subequations}
By \cite[Lem.~3.1]{AltZ20}, this problem yields a unique solution pair $(x, \nu)$ with~$\nu$ being a Lagrange multiplier we are not particularly interested in. The definition of~$\varphi_1$ implies that $\tau\varphi_1(-\tau\calAo)\, h = - \big[ \varphi_{0}({-\tau \calAo}) - \id \big]\, \calAo^{-1}\, h$ for all $h\in\cHo$. Hence, the exponential Euler scheme can be rewritten as 
\begin{align*}
	u^{n+1} 
	= \calB^-g(t^{n+1}) + \varphi_{0}(-\tau\calA_{\ker}) \big(u^n - \calB^-g(t^n) - w^n\big) + w^n,
\end{align*}
where the auxiliary variable~$w^n\in \Vo$ is defined as the solution of
\begin{subequations}
	\label{eq:wn}	
	\begin{alignat}{4}
		&\calA w^n&\ +\ &\calB^\ast \nu^n\ &=&\ f(t^n, u^n) - \calB^- \dot g(t^n)  &&\qquad\text{in } \V^* ,\\
		&\calB w^n& & &=&\ 0 &&\qquad\text{in } \Q^*.
	\end{alignat}
\end{subequations}
Similarly as before, $\nu^n$ acts as Lagrange multiplier corresponding to
the constraint. To compute the action of $\varphi_{0}(-\tau\calA_{\ker})$, one can consider the corresponding PDAE formulation. This leads to
\begin{equation}\label{eq:Exp_Euler}
    u^{n+1}=\calB^-g(t^{n+1})+z(t^{n+1})+w^n, 
\end{equation}
where $z$ is the solution of the linear and homogeneous system 
\begin{subequations}
\label{eq:zn}	
\begin{alignat}{5}
	\dot{z}(t)&\ +\ &\calA z(t)&\ +\ &\calB^\ast \mu(t)\, &= 0 &&\qquad\text{in } \V^*, \label{eq:zn:a}	\\
	& &\calB z(t)& & &= 0 &&\qquad\text{in } \Q^* \label{eq:zn:b}	
\end{alignat}
\end{subequations}
on the time interval $[t^n,t^{n+1}]$ with initial condition~$z(t^n) = u^n-\calB^-g(t^n)-w^n$. Here, $\mu(t)$ is again a Lagrange multiplier. 

Adding perturbations to this type of method is a bit more delicate than for the previously introduced methods. Moreover, there are various possible ways to do this ensuring that the model constraint is satisfied. We discuss one approach in detail. 

One possibility to incorporate perturbations is to add them directly to~\eqref{eq:Exp_Euler}, i.e., to consider 
\begin{align*}
	U^{n+1} 
	= \calB^-g(t^{n+1}) + z(t^{n+1}) + w^n + \xi^n(\tau).
\end{align*}
with $\xi^n(\tau)\in \Vo$. In practice, this would mean that $\xi^n(\tau)$ equals the solution of the saddle point problem 
\begin{subequations}
\label{eq:xin}	
\begin{alignat}{4}
	&\calA \xi^n(\tau)&\ +\ &\calB^\ast \nu^n\ &=&\ \sigma\tau^{p+1/2}\calA\, \tilde{\xi}^n &&\qquad\text{in } \V^* ,\\
	&\calB \xi^n(\tau)& & &=&\ 0 &&\qquad\text{in } \Q^*,
\end{alignat}
\end{subequations}
where $\sigma$ is again the scaling factor of the noise and the $\tilde{\xi}^n$ are i.i.d.~standard normal random variables i.e., $\tilde{\xi}^n \sim\calN(0,1)$. The unique solvability of system \eqref{eq:xin} follows again from~\cite[Lem.~3.1]{AltZ20}. We summarize the necessary steps in Algorithm~\ref{alg:exp_euler_variant_a}.
\begin{algorithm}
    \caption{Probabilistic exponential Euler scheme}
    \label{alg:exp_euler_variant_a}
    \begin{algorithmic}[1]
        \State \textbf{Input}: step size $\tau$, consistent initial data $U^0=u^0\in \V$, right-hand sides $f$, $g$, problem-independent scaled noise $\sigma$ \vspace{0.5em}
        \For{$n=0$ \textbf{to} $N-1$}
            \State compute $\mathcal{B}^{-}g(t^n)$, $\mathcal{B}^{-}g(t^{n+1})$, and $\mathcal{B}^{-}\dot{g}(t^{n})$ by~\eqref{eqn:Binvers}
            \State compute $w^n$ by~\eqref{eq:wn}
            \State compute $z$ as the solution of \eqref{eq:zn} on $[t^n,t^{n+1}]$ with initial data $U^n - \mathcal{B}^{-}g(t^n) - w^n$
            \State compute $\xi^n(\tau)$ as the solution of \eqref{eq:xin} with $\sigma$, $p=1$, and i.i.d.~$\tilde{\xi}^n \sim \mathcal{N}(0,1)$
            \State set $U^{n+1} = \mathcal{B}^{-}g(t^{n+1}) + z(t^{n+1}) + w^n + \xi^n(\tau)$
        \EndFor
    \end{algorithmic}
\end{algorithm}
\begin{remark}\label{Rem5.1}
A single step of the probabilistic exponential Euler scheme involves the solution of four stationary and one transient saddle point problems with only one evaluation of the nonlinear function $f$. Note that all these systems are linear and that the time-dependent system is homogeneous, allowing it to be solved without further regularization.
\end{remark}
%
%
\subsection{Probabilistic exponential integrator of second order}
In this part, we analyze a probabilistic exponential integrator of second order for constrained parabolic systems. As underlying deterministic integrator, we consider  
\begin{equation}\label{eq:exp_second}
    u^{n+1} = u^{n+1}_\text{Eul} + \tau\varphi_2(-\tau \calA_{\ker}) \Big[ f(t^{n+1}, u^{n+1}_{\text{Eul}})-\calB^{-}\dot{g}(t^{n+1})-f(t^n,u^n)+\calB^-\dot{g}(t^{n})\Big]
\end{equation}
as introduced in~\cite{AltZ20}. Therein, $u_{\text{Eul}}^{n+1}$ denotes the result of one step of the deterministic exponential Euler scheme~\eqref{eq:Exp_Euler}. 
Considering only the part in the kernel, we get 
\[
    \Psi_{\tau}(t,v) 
    = \Psi_{\tau,\text{Eul}} (t, v) + \tau\varphi_2(-\tau \calA_{\ker}) \Big[ f(t+\tau,\Psi_{\tau,\text{Eul}} (t, v))-\calB^{-}\dot{g}(t+\tau)-f(t,v)+\calB^-\dot{g}(t)\Big]
\]
where $\Psi_{\tau,\text{Eul}}(t\,v)$ is given by \eqref{eq:PsiExpEul}. It is shown in~\cite{AltZ20} that this method is of order~$q=2$. As before, we discuss the Lipschitz condition in the following lemma. 
\begin{lemma}
Suppose that Assumptions~\ref{assB}, \ref{assA}, and~\ref{assf} hold and let $\calA$ be symmetric. Then the exponential integrator of second order satisfies the Lipschitz property~\eqref{eq:lipPsi}. 
\end{lemma}
\begin{proof}
Direct calculations together with Lemma~\ref{Lem:LipExpEul} show 
\begin{align*}
    \big\|&\Psi_{\tau}(t,v_1)-\Psi_{\tau}(t,v_2)\big\|_{\cH} \\
    &\qquad\le \big\|\Psi_{\tau,\text{Eul}}(t,v_1)-\Psi_{\tau,\text{Eul}}(t,v_2) \big\|_{\cH} + \tau\, \big\|\varphi_2(-\tau\calA_{\ker})\, \big(f(t+\tau,\Psi_{\tau,\text{Eul}}(t,v_1)) \\ 
    &\hspace{5.8cm} - f(t+\tau, \Psi_{\tau,\text{Eul}}(t,v_2)) \big) - \big(f(t,v_1)-f(t,v_2)\big) \big\|_{\cH} \\
    &\qquad\le \big( 1 + \tau \tilde L_{\Psi} \big)\, \big\| v_1 - v_2 \big\|_{\cH}
\end{align*}
for some $\tilde L_{\Psi}>0$, where we have used again the boundedness of the $\varphi$-functions and the symmetry of $\calA$. 
\end{proof}

Similar to the first-order exponential Euler method, the practical implementation of this method involves the solution of several saddle point problems. By the recursion formula \eqref{eq:varphi}, the approximation at time $t^{n+1}$ is defined by 
\begin{equation}
    u^{n+1} 
    = u^{n+1}_{\text{Eul}}+z(t^{n+1})-\widehat{w}^n+\overline{w}^n,
\end{equation}
where $\overline{w}^n$ is computed as the solution of the stationary problem
\begin{subequations}
	\label{eq:overline_wn}	
	\begin{alignat}{4}
		&\calA \overline{w}^n&\ +\ &\calB^\ast \overline{\nu}^n\ &=&\ f(t^{n+1},u^{n+1}_{\text{Eul}})-\calB^{-}\dot{g}(t^{n+1})-f(t^n,u^n)+\calB^-\dot{g}(t^{n}) &&\qquad\text{in } \V^* ,\\
		&\calB \overline{w}^n& & &=&\ 0 &&\qquad\text{in } \Q^*,
	\end{alignat}
\end{subequations}
and $\widehat{w}^n$ as the solution of 
\begin{subequations}
	\label{eq:hat_wn}	
	\begin{alignat}{4}
		&\calA \widehat{w}^n&\ +\ &\calB^\ast \widehat{\nu}^n\ &=&\ \tfrac{1}{\tau}\,\overline{w}^n &&\qquad\text{in } \V^* ,\\
		&\calB \widehat{w}^n& & &=&\ 0 &&\qquad\text{in } \Q^*.
	\end{alignat}
\end{subequations}
Moreover, $z(t^{n+1})$ is the solution of the homogeneous system \eqref{eq:zn} on the time interval $[t^{n},t^{n+1}]$ with initial value $z(t^n)=\widehat{w}^n$. The Lagrange multipliers $\overline{\nu}^n$ and $\widehat{\nu}^n$ are once more not of particular interest and simply serve as dummy variables. For more details and the corresponding convergence proof, we refer to~\cite[Sect.~4]{AltZ20}. 

For the construction of a randomized exponential integrator of second order, we follow the approach from the previous section but with $p=2$. This then leads to the probabilistic integrator shown in Algorithm~\ref{alg:exp_sec_variant_a}.  
\begin{algorithm}
    \caption{Probabilistic exponential integrator of second order }
    \label{alg:exp_sec_variant_a}
    \begin{algorithmic}[1]
    \State \textbf{Input}: step size $\tau$, consistent initial data $U^0=u^0\in \V$, right-hand sides $f$, $g$, problem-independent scaled noise~$\sigma$
        \vspace{0.5em}
        \For{$n=0$ \textbf{to} $N-1$}
        \State compute one step of the deterministic exponential Euler method~$U^{n+1}_{\text{Eul}}$ given by \eqref{eq:Exp_Euler}
        \State compute $\calB^{-}\dot{g}(t^{n})$ and $\calB^{-}\dot{g}(t^{n+1})$ by~\eqref{eqn:Binvers}
        \State compute $\overline{w}^n$ by~\eqref{eq:overline_wn}
        \State compute $\widehat{w}^n$ by~\eqref{eq:hat_wn} 
        \State compute $z$ as the solution of \eqref{eq:zn} on $[t^n,t^{n+1}]$ with initial data $\widehat{w}^n$
        \State compute $\xi^n(\tau)$ as the solution of~\eqref{eq:xin} with $\sigma$, $p=2$, and i.i.d.~$\tilde{\xi}^n \sim\calN(0,1)$
        \State set $U^{n+1} 
	= U^{n+1}_{\text{Eul}} + z(t^{n+1}) - \widehat{w}^n +\overline{w}^n + \xi^n(\tau)$
        \EndFor
    \end{algorithmic}
\end{algorithm}
%
%
\subsection{Numerical example}
This part is devoted to the numerical validation of the convergence result presented in Theorem \ref{thm:meanErr}. In this test case, we investigate the mean square convergence rates of the probabilistic integrators introduced in the previous subsections when applied to the semi-linear heat equation with a constraint on the integral mean; see Example~\ref{ex:PDAE}. We prescribe homogeneous Dirichlet boundary conditions, i.e., $u(t, 0) = u(t, 1) = 0$, set the right-hand side to $g(t) = t$, and initialize the system consistently by 
\[
    u^0(x) 
    = \sin(2\pi x)^3,\qquad 
    \calB u^0 
    = \int_0^1 u^0(x) \sin(\pi x) \dx
    = 0 
    = g(0).
\]
For the numerical solution, we have used a second-order finite difference approximation of the Laplacian with $100$ grid points. As time interval, we consider~$[0,T] = [0, 0.1]$. We vary the value of~$p$ in Assumption \ref{ass:noise}, namely $p\in\{1/2,1,3/2\}$ for the first-order methods and $p\in\{1/2,1,3/2,2\}$ for the second-order methods. Then, we compute~$1000$ trajectories of the numerical solution with a fixed noise scale $\sigma= 4$ and compute the approximate mean square order of convergence for the implicit Euler method, the midpoint scheme, 
and the exponential integrators. 
The results are depicted in Figure \ref{Fig:order:all:PDAE} and show the predicted $\min\{p,q\}$-order convergence.  
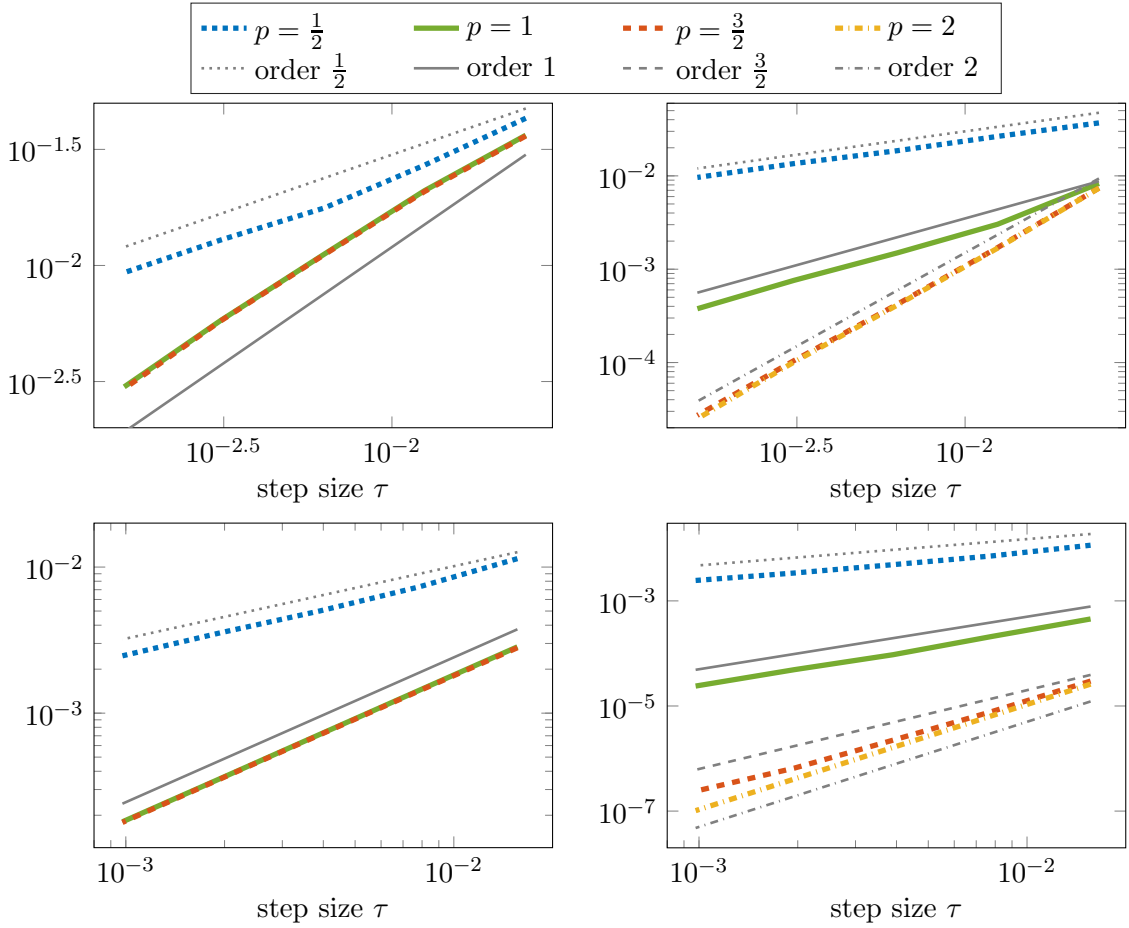
\begin{figure}
\centering
%
%
\begin{tikzpicture}

\begin{axis}[%
width=2.4in,
height=1.7in,
scale only axis,
xmode=log,
xlabel={step size $\tau$},
xmin=0.0013,
xmax=0.03,
xminorticks=true,
ymode=log,
ymin=0.002,
ymax=0.05,
yminorticks=true,
axis background/.style={fill=white},
legend style={at={(1.98,1.31)}, legend cell align=left, align=left, draw=white!15!black},
legend columns=4
]


\addplot [color=mycolor1, dotted, line width=2.0pt]
  table[row sep=crcr]{%
	0.025	0.043074\\
	0.0125	0.027046\\
	0.0063	0.017695\\
	0.0031	0.012879\\
	0.0016	0.0093137\\
};
\addlegendentry{$p=\frac12$\qquad} 

\addplot [color=mycolor5, line width=2.0pt]
  table[row sep=crcr]{%
0.025	0.0364\\
0.0125	0.021\\
0.0063	0.0112\\
0.0031	0.0058\\
0.0016	0.003\\
};
\addlegendentry{$p=1$\qquad} 

\addplot [color=mycolor2, dashed, line width=2.0pt]
table[row sep=crcr]{%
0.025	0.036124\\
0.0125	0.020813\\
0.0063	0.011141\\
0.0031	0.0057755\\
0.0016	0.0029453\\
};
\addlegendentry{$p=\frac32$\qquad} 

\addplot [color=mycolor3, dashdotted, line width=2.0pt]
table[row sep=crcr]{%
1 1\\
1 1\\
};
\addlegendentry{$p=2$\ }

\addplot [color=gray, dotted, line width=1pt]
table[row sep=crcr]{%
0.025	0.0474341649025257\\
0.0016	0.012\\
};
\addlegendentry{order $\frac12$\qquad} 

\addplot [color=gray, line width=1pt]
table[row sep=crcr]{%
0.025	0.03\\
0.0016	0.00192\\
};
\addlegendentry{order $1$\qquad } 

\addplot [color=gray, dashed, line width=1pt]
table[row sep=crcr]{%
	1 1\\
	1 1\\
};
\addlegendentry{order $\frac32$\qquad } 

\addplot [color=gray, dashdotted, line width=1pt]
table[row sep=crcr]{%
	1 1\\
	1 1\\ 
};
\addlegendentry{order $2$\ } 

\end{axis}


\begin{axis}[%
	width=2.4in,
	height=1.7in,
	at={(3.0in,0.in)},
	scale only axis,
	xmode=log,
	xlabel={step size $\tau$},
	xmin=0.0013,
	xmax=0.03,
	xminorticks=true,
	ymode=log,
	ymin=2e-5,
	ymax=0.06,
	yminorticks=true,
	axis background/.style={fill=white},
	]
			
	\addplot [color=mycolor1, dotted, line width=2.0pt]
	table[row sep=crcr]{%
	0.025	0.036929\\
	0.0125	0.026492\\
	0.0063	0.018631\\
	0.0031	0.013541\\
	0.0016	0.0096227\\
	};
	
	\addplot [color=mycolor5, line width=2.0pt]
	table[row sep=crcr]{%
	0.025	0.0082972\\
	0.0125	0.0030331\\
	0.0063	0.0014988\\
	0.0031	0.00075946\\
	0.0016	0.00037716\\
	};
	
	\addplot [color=mycolor2, dashed, line width=2.0pt]
	table[row sep=crcr]{%
	0.025	0.0074104\\
	0.0125	0.0017047\\
	0.0063	0.00042257\\
	0.0031	0.00010485\\
	0.0016	2.7209e-05\\
	};
	
\addplot [color=mycolor3, dashdotted, line width=2.0pt]
table[row sep=crcr]{%
	0.025	0.0073867\\
	0.0125	0.0016835\\
	0.0063	0.00041405\\
	0.0031	0.00010113\\
	0.0016	2.5038e-05\\
};

	
	\addplot [color=gray, dotted, line width=1pt]
	table[row sep=crcr]{%
		0.025	0.0474341649025257\\
		0.0016	0.012\\
	};
	
	\addplot [color=gray, line width=1pt]
	table[row sep=crcr]{%
	0.025	0.00875\\
	0.0016	0.00056\\
	};
	
	\addplot [color=gray, dashdotted, line width=1pt]
	table[row sep=crcr]{%
		0.025	0.009375\\
		0.0016	3.84e-05\\
	};
			
\end{axis}


\begin{axis}[%
	width=2.4in,
	height=1.7in,
	at={(0.0in,-2.2in)},
	scale only axis,
	xmode=log,
	xlabel={step size $\tau$},
	xmin=0.0008,
	xmax=0.02,
	xminorticks=true,
	ymode=log,
	ymin=1.2e-04,
	ymax=0.02,
	yminorticks=true,
	axis background/.style={fill=white},
	]
	
\addplot [color=mycolor1, dotted, line width=2.0pt]
table[row sep=crcr]{%
0.015625	0.011422\\
0.0078125	0.0073044\\
0.0039062	0.0050095\\
0.0019531	0.0035597\\
0.00097656	0.0024686\\
};

\addplot [color=mycolor5, line width=2.0pt]
table[row sep=crcr]{%
0.015625	0.0028437\\
0.0078125	0.001428\\
0.0039062	0.00071623\\
0.0019531	0.00035878\\
0.00097656	0.00018037\\
};

\addplot [color=mycolor2, dashed, line width=2.0pt]
table[row sep=crcr]{%
0.015625	0.0028034\\
0.0078125	0.0014161\\
0.0039062	0.00071141\\
0.0019531	0.00035663\\
0.00097656	0.00017854\\
};

\addplot [color=gray, dotted, line width=1pt]
table[row sep=crcr]{%
0.015625	0.0126491106406735\\
0.00097656	0.0032\\
};

\addplot [color=gray, line width=1pt]
table[row sep=crcr]{%
0.015625	0.00375\\ 
0.00097656	0.00024\\
};
	
\end{axis}


\begin{axis}[%
	width=2.4in,
	height=1.7in,
	at={(3.0in,-2.2in)},
	scale only axis,
	xmode=log,
	xlabel={step size $\tau$},
	xmin=0.0008,
	xmax=0.02,
	xminorticks=true,
	ymode=log,
	ymin=2e-08,
	ymax=0.03,
	ytick={1e-3,1e-5,1e-7},
	yminorticks=true,
	axis background/.style={fill=white},
	]
	
	\addplot [color=mycolor1, dotted, line width=2.0pt]
	table[row sep=crcr]{%
	0.015625	0.011462\\
	0.0078125	0.0071766\\
	0.0039062	0.0048922\\
	0.0019531	0.0033923\\
	0.00097656	0.0024471\\
	};
	
	\addplot [color=mycolor5, line width=2.0pt]
	table[row sep=crcr]{%
	0.015625	0.00045309\\
	0.0078125	0.0002104\\
	0.0039062	9.4584e-05\\
	0.0019531	4.91e-05\\
	0.00097656	2.3947e-05\\
	};
	
	\addplot [color=mycolor2, dashed, line width=2.0pt]
	table[row sep=crcr]{%
	0.015625	2.981e-05\\
	0.0078125	7.8327e-06\\
	0.0039062	2.2381e-06\\
	0.0019531	6.574e-07\\
	0.00097656	2.396e-07\\
	};
	
	\addplot [color=mycolor3, dashdotted, line width=2.0pt]
	table[row sep=crcr]{%
	0.015625	2.6081e-05\\
	0.0078125	6.5175e-06\\
	0.0039062	1.6298e-06\\
	0.0019531	4.0743e-07\\
	0.00097656	1.0198e-07\\
	};
	
	
	\addplot [color=gray, dotted, line width=1pt]
	table[row sep=crcr]{%
	0.015625	0.01875\\ 
	0.00097656	0.00468749399999616\\
	};
	
	\addplot [color=gray, line width=1pt]
	table[row sep=crcr]{%
		0.015625	0.00078125\\
		0.00097656	4.8828e-05\\
	};
	
	\addplot [color=gray, dashed, line width=1pt]
	table[row sep=crcr]{%
	0.015625	3.90625e-05\\ 
	0.00097656	6.103492187515e-07\\
	};
	
	\addplot [color=gray, dashdotted, line width=1pt]
	table[row sep=crcr]{%
	0.015625	1.220703125e-05\\ 
	0.00097656	4.768347168e-08\\
	};
	
\end{axis}

\end{tikzpicture}%
    \caption{Convergence history for the mean square error for the four probabilistic time integrators introduced in Section~\ref{Sec.PNM}: implicit Euler method (upper left), midpoint scheme (upper right), exponential Euler (lower left), and the second-order exponential integrator (lower right). Numerical results for different values of $p$ and fixed noise scale $\sigma = 4$. }
    \label{Fig:order:all:PDAE} 
\end{figure}
%
%
\section{Calibrating Forward Uncertainty Propagation}\label{Sec.Calib}
This section aims to discuss one potential approach for choosing $\sigma$ when employing randomized time integrators to solve constrained parabolic systems. Let us consider once more the constrained FitzHugh--Nagumo model from Examples~\ref{ex:FitzHugh} and~\ref{ex:FitzHugh:revisit}. Through this example, we can observe that the scale parameter $\sigma$ controls the apparent uncertainty in the solver; see Figure~\ref{Fig:pert:Fitz:ds}. While the random draws contract towards the true solution, the scale parameter $\sigma$ affects the error magnitude introduced by the probabilistic numerical method. 
To overcome this drawback, we follow~\cite{ConGSSZ17} and calibrate $\sigma$ to match the error magnitude given by classical error indicators. Let us denote the approximation given by the probabilistic solver (with step size~$\tau$) by~$U^n_{\tau,\sigma}$ and the corresponding deterministic approximation by~$u^n_{\tau,0}$. The simplest error indicator might be the difference between solutions obtained with different step sizes, i.e., $E^n\coloneqq u^n_{\tau,0}-u^{2n}_{\tau/2,0}$.

In line with~\cite[Sect.~3]{ConGSSZ17}, we study a probability distribution denoted  by $\pi(\sigma)$, which achieves maximum likelihood when the desired matching occurs. This estimation of scale matching involves the comparison of a Gaussian approximation  $\tilde{\mu}^n_{\tau, \sigma} = \mathcal{N}(\mathbb{E}(U^n_{\tau,\sigma}), \mathbb{V}(U^n_{\tau,\sigma}))$ and the natural Gaussian measure represented as $\nu^n_{\tau,0} = \mathcal{N}(u^n_{\tau,0}, (E^n)^2)$ at each time step. Here, $ \mathbb{V}(U^n_{\tau,\sigma})$ denotes the variance of the vector-valued random variables $U^n_{\tau,\sigma}$. Using the Bhattacharyya distance~\cite{kail1967} to measure the disparity between these two normal distributions, the probability distribution $\pi(\sigma)$ can be constructed by 
\[
    \pi(\sigma) 
    \propto \prod_n \exp \left( -d(\tilde{\mu}^n_{\tau, \sigma}, \nu^n_{\tau,0}) \right).
\]
We use the following (component-wise) formulae  
\begin{equation}\label{eq:Bdistance}
  d(\tilde{\mu}^n_{\tau, \sigma}, \nu^n_{\tau,0}) 
    = \frac{1}{4} \frac{\left(\mathbb{E}(U^n_{\tau,\sigma})-u^n_{\tau,0}\right)^2}{\mathbb{V}(U^n_{\tau,\sigma})+(E^n)^2} 
    + \frac{1}{4}\ln\left(\frac{1}{4}\left(\frac{\mathbb{V}(U^n_{\tau,\sigma})}{(E^n)^2}+\frac{(E^n)^2}{\mathbb{V}(U^n_{\tau,\sigma})}+2\right)\right)  
\end{equation}
for the Bhattacharyya distance when applied to the case of two univariate normal distributions~\cite{lee2012}. Given a randomized solution, calculating the mean and variance is straightforward from their definitions~\cite{casella2024}. Further note that equation~\eqref{eq:Bdistance} should be understood component-wise.  

We seek the value of \(\sigma\) that maximizes the function \(\pi(\sigma)\). To achieve this, we employ the MATLAB optimization algorithm \texttt{fminsearch}. By minimizing the negative of \(\pi(\sigma)\), we reformulate the problem as 
\[
\sigma^* = \arg\min_{\sigma} -\pi(\sigma).
\]
By this choice, we depict the true trajectories of the $V$ species of the constrained FitzHugh--Nagumo model and $100$ realizations from the probabilistic implicit Euler method in Figure~ \ref{fig:FitzHugh:sigma}.
Therein, we use various step sizes and a fixed noise scale $\sigma=0.5$ for comparison. The plots clearly show that the random realizations from the measure associated with the calibrated probabilistic implicit Euler solver contract towards the true solution faster than those obtained by probabilistic implicit Euler solver with a fixed $\sigma$. 
In Figure~\ref{fig:FitzHugh:variations} we present a comparison of the observed marginal variations and their mean in the calibrated probabilistic implicit Euler method and the error indicator. This comparison demonstrates a good agreement in the marginal variances. 
Algorithm~\ref{alg:calibration} describes the process for calibrating the noise scale $\sigma$ in probabilistic numerical methods.
\begin{figure}
  \centering
        \includegraphics[width=0.5\textwidth]{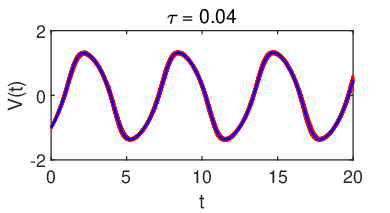}\hfill
        \includegraphics[width=0.5\textwidth]{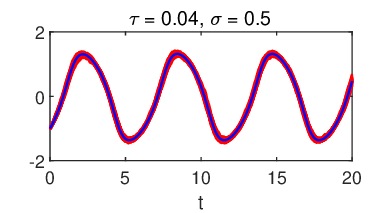}\vfill
        \includegraphics[width=0.5\textwidth]{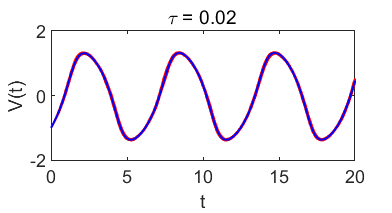}\hfill
        \includegraphics[width=0.5\textwidth]{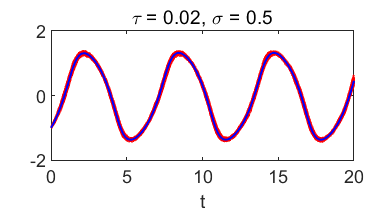}\vfill
        \includegraphics[width=0.5\textwidth]{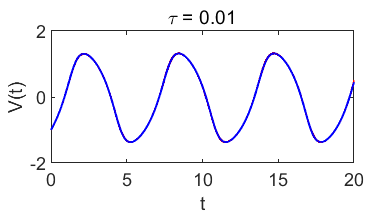}\hfill
        \includegraphics[width=0.5\textwidth]{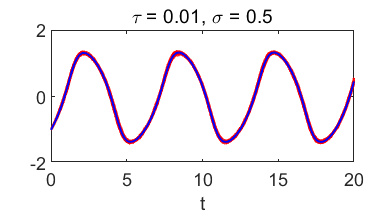}\vfill
\caption{The exact trajectory of $V$ (blue) and $100$ realizations (red) computed by the probabilistic implicit Euler method with noise scale $\sigma^*$ (left) and fixed noise scale $\sigma=0.5$ (right). }
    \label{fig:FitzHugh:sigma} 
\end{figure}
\begin{figure}
  \centering
        \includegraphics[width=1\textwidth]{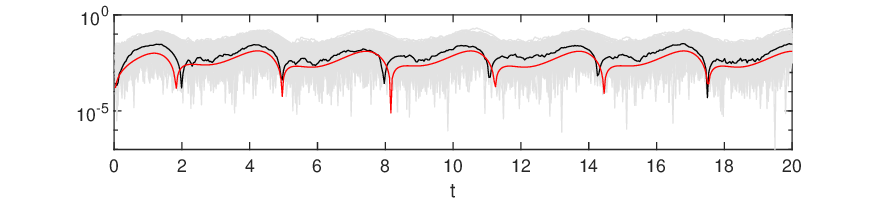}  
\caption{A comparison of the observed variations (gray) and their mean (black) in the calibrated probabilistic implicit Euler method using $\sigma^*$ and $\tau=0.04$ for the constrained FitzHugh--Nagumo model from Example~\ref{ex:FitzHugh:revisit}. The error indicator is marked in red.}
\label{fig:FitzHugh:variations} 
\end{figure}
\begin{algorithm}
    \caption{Calibration of noise scale $\sigma$}
    \label{alg:calibration}
    \begin{algorithmic}[1] 
    \State \textbf{Input}: step size $\tau$, consistent initial data $U^0 = u^0 \in \mathcal{V}$, right-hand sides $f$, $g$, initial scaled noise $\sigma$, number of realizations $M$
    \vspace{0.5em}
    \For{$n=0$ \textbf{to} $N-1$}
        \State compute deterministic solutions $u^n_{\tau,0}$ and $u^{2n}_{\tau/2,0}$ 
        \State compute error indicator $E^n \coloneqq u^n_{\tau,0} - u^{2n}_{\tau/2,0}$
        \State compute random solution $U^n_{\tau,\sigma}$ using a probabilistic solver with $M$ realizations
         \State compute the mean~$\mathbb{E}(U^n_{\tau,\sigma})$ and the variance~$\mathbb{V}(U^n_{\tau,\sigma})$ of the random solutions
        \State compute the Bhattacharyya distance $d(\tilde{\mu}^n_{\tau, \sigma}, \nu^n_{\tau,0})$ as in~\eqref{eq:Bdistance}
        \State set $\sigma = \arg\min_\sigma -\pi(\sigma)$
    \EndFor
    \end{algorithmic}
\end{algorithm}
%
%
\section{Conclusion}\label{sect:conclusion}
In this paper, we have introduced and analyzed a class of probabilistic time integrators of first and second order for the numerical solution of semi-linear parabolic equations with linear constraints. Such methods incorporate random perturbations, which do not affect the constraint manifold, in order to quantify the uncertainty induced by the time integration scheme. In terms of convergence, we have studied the mean square error and showed that a balanced inclusion of perturbations does not reduce the convergence order known from the deterministic setting. More precisely, the overall order equals the minimum of the deterministic convergence order and the order of the introduced statistical error. All results are validated through numerical experiments. 
%
%
%
%
\bibliographystyle{alpha} 
\bibliography{bib_probInt}
\appendix
\end{document}